\documentclass[10pt]{amsart}
\usepackage[cp1251]{inputenc}
\usepackage[english,russian]{babel}
\usepackage{amsmath}
\usepackage{longtable}
\usepackage{tikz}
\usepackage{amssymb}
\usepackage{amsfonts}

\newtheorem{lemma}{Lemma}
\newtheorem{theorem}{Theorem}

\begin{document}
\renewcommand{\refname}{References}
\renewcommand{\proofname}{Proof.}
\thispagestyle{empty}
\title[On strictly Deza graphs]{On strictly Deza graphs \\ with parameters $\bf (n,k,k-1,a)$ }
\author{{V.V. Kabanov, N.V. Maslova, L.V. Shalaginov}}
\address{Vladislav Vladimirovich Kabanov
\newline\hphantom{iii} Krasovskii Institute of Mathematics and Mechanics,
\newline\hphantom{iii} 16, S. Kovalevskaya str.,
\newline\hphantom{iii} 630090, Yekaterinburg, Russia}
\email{vvk@imm.uran.ru}
\address{Natalia~Vladimirovna~Maslova
\newline\hphantom{iii} Krasovskii Institute of Mathematics and Mechanics,
\newline\hphantom{iii} 16, S. Kovalevskaya str.
\newline\hphantom{iii} 620990, Yekaterinburg, Russia
\newline\hphantom{iii} Ural Federal University,
\newline\hphantom{iii} 19, Mira Str.
\newline\hphantom{iii} 620002, Yekaterinburg, Russia}
\email{butterson@mail.ru}
\address{Leonid Viktorovich Shalaginov
\newline\hphantom{iii} Chelyabinsk State University,
\newline\hphantom{iii} 129, Bratiev Kashirinykh Str.
\newline\hphantom{iii} 454001, Chelyabinsk, Russia
\newline\hphantom{iii} Krasovskii Institute of Mathematics and Mechanics,
\newline\hphantom{iii} 16, S. Kovalevskaya str.
\newline\hphantom{iii} 620990, Yekaterinburg, Russia}
\email{44sh@mail.ru}

\thanks{The first author was supported by RFBR according to the research projects 17-51-560008,
the second author was supported by the Complex Program of Ural Branch Russian Academy of Sciences, project 18-1-1-17, 
the third author was supported by RFBR according to the research projects  16-31-00316 and 17-51-560008.}
\maketitle {\small
\begin{quote}

\noindent{\sc Abstract. }
A nonempty $k$-regular graph $\Gamma$ on $n$ vertices is called a {\em Deza graph} if there exist
constants $b$ and $a$ $(b \geq a)$ such that any pair of distinct vertices of $\Gamma$ has either $b$ or $a$ common neighbours.
The~quantities $n$, $k$, $b$, and $a$ are called the parameters of $\Gamma$ and are written as the quadruple $(n,k,b,a)$.
If a Deza graph has diameter 2 and is not strongly regular, then it is called {\em a strictly Deza graph}. In the present paper, we investigate strictly  Deza graphs whose parameters  $(n, k, b, a) $ satisfy  the conditions $k = b + 1$ and
$\displaystyle\frac{k(k - 1) - a(n - 1)}{b - a} > 1$.

\medskip

 \end{quote}
}

\bigskip
\hfill {\em  Dedicated to the memory of Michel Deza}
\medskip

\section{Introduction}

In the present paper, we consider finite undirected graphs without loops and multiple edges.
A nonempty $k$-regular graph $\Gamma$ on $n$ vertices is called a {\em Deza graph} if there exist
constants $b$ and $a$ such that any pair of distinct vertices of $\Gamma$ has either $b$ or $a$ common neighbours.
We assume further that $b \geq a$.
The quantities $n$, $k$, $b$, and $a$ are called the parameters of $\Gamma$ and are written as the quadruple $(n,k,b,a)$.

The concept of a Deza graph was introduced in 1999  by M. Erickson, S. Fernando, W. Haemers, D. Hardy, and J. Hemmeter in the seminal paper  \cite{EFHHH} influenced by A. Deza and M. Deza \cite{dd}.   Deza graphs generalize strongly regular graphs in the sense that the number of common neighbours of any pair of vertices in a Deza graph does not depend on adjacency.

A strongly regular graph has diameter $2$, except for the trivial case of a disjoint union of complete graphs. As opposed to  strongly regular graphs,  Deza graphs can have diameter greater than 2.
If a Deza graph has diameter 2 and is not strongly regular, then it is called {\em a strictly Deza graph}. So, we have a trihotomy for the class of Deza graphs: strongly regular graphs, strictly Deza graphs, and Deza graphs of diameter greater than 2.

In  \cite{EFHHH} a basic theory of strictly Deza graphs was developed and several ways to construct such graphs were introduced.
Moreover, all strictly Deza graphs with number of vertices at most 13 were found. In 2011, the investigation of strictly Deza graphs was continued by S. Goryainov and L. Shalaginov in \cite{gsh1}. They found all strictly Deza graphs whose number of vertices is equal to 14, 15, or 16.  In 2014, S. Goryainov and L. Shalaginov in \cite{gsh2} found all strictly Deza graphs  that are Cayley graphs with number of vertices less than 60.

Problems arising in the theory of strictly Deza graphs  sometimes are  similar to problems in the theory of strongly regular graphs.
However, results and methods in these theories differ. In our opinion, an analysis of these differences can enrich
both theories.

For example, it is known that the connectivity of a connected strongly regular graph equals its valency \cite{vc}.  In 2014, the connectivity of some strictly Deza graphs was investigated in \cite{ggk}. In particular,  an example of a strictly Deza graph whose connectivity and valency were not equal was  found.

If $\Gamma$ is a strongly regular graph, then its parameters are written as $(n, k, \lambda, \mu)$, where  $\lambda$ is the number of common neighbours of every two adjacent vertices of $\Gamma$ and $\mu$ is the number of common neighbours of every two nonadjacent and distinct vertices of $\Gamma$.

If a strongly regular graph $\Gamma$ has parameters $(n, k, \lambda, \mu)$ such that $k = \mu$, then 
$\Gamma$ is a complete  multipartite graph with parts of size $n - k$ (see Section 1.3 in \cite{BCN}).
An analogue of this result for strictly Deza graphs with condition $k = b$ was also obtained  in the above-mentioned paper  \cite{EFHHH}.

The complement of a strongly regular graph with parameters $(n, k, \lambda, \mu)$ is also strongly regular with parameters
$(v, v-k-1, v-2k+\mu-2, v-2k+\lambda)$. Therefore, if a strongly regular graph $\Gamma$ has parameters $(n, k, \lambda, \mu)$, where $k = \mu$, then the parameters of complement $\overline{\Gamma}$ satisfy the equality $\overline{k} = \overline{\lambda} + 1$. Hence, the structure of a strongly regular graph $\Gamma$ with $k = \lambda + 1$ can be obtained from the corresponding result for a strongly regular  graph with
$k = \mu$ and vice versa.

It is important to note that there is  the other  situation in the case of strictly Deza graphs. Namely, let $\Gamma$ be a strictly Deza graph. Its complement $\overline\Gamma$ is a Deza graph only if $\Gamma$ is a coedge-regular graph with  $b = a + 2$. Thus, there is no direct connections between a strictly Deza graph with parameters satisfying $k = b$ and a strictly Deza graph with parameters satisfying $k = b+1$. The aim of this paper is to investigate  strictly  Deza graphs with parameters $(n, k, b, a)$ satisfying the condition $k = b + 1$ which resemble to strongly regular graphs with $k = \lambda + 1$.
The structure of such Deza graphs turned out to be much more complicated than the corresponding case of strongly regular graphs.

Let us introduce some definitions and notation.

Let $\Gamma$ be a graph with the vertex set $V(\Gamma)$, and let $v\in V(\Gamma)$. The set of vertices adjacent to $v$ is called the {\em neighbourhood of} $v$ and is denoted by $ N(v) $.  The set $N(v) \cup \{v\}$ is called the {\em closed neighbourhood of} $v$ and is denoted by $N[v]$. The set of vertices at distance 2 from a vertex $v$ is called {\em the second neighbourhood} of $v$ and  is denoted by $ N_2(v) $.

Let $\Delta_1$ and $\Delta_2$ be graphs. A graph $\Gamma$ is called the  {\em extension of $\Delta_1$ by $\Delta_2$} if the following conditions hold:

(1) $V(\Gamma)$ is the set of pairs $(v_1, v_2)$ such that $v_1\in V(\Delta_1)$ and  $v_2\in V(\Delta_2)$.

(2) Vertices $(v_1, v_2)$ and $(u_1, u_2)$ are adjacent in $\Gamma$ if and only if either $v_1$ and $u_1$ are adjacent in $\Delta_1$ or $v_1 = u_1$ and $v_2$, $u_2$ are adjacent in $\Delta_2$.

Sometimes, such a graph $\Gamma$ is called the {\em composition of $\Delta_1$ and $\Delta_2$}.

We say that $\Gamma$ is the {\em $m$-clique extension of $\Delta_1$} if $\Delta_2$ is the complete graph $K_m$ on $m$ vertices.
We say that $\Gamma$ is  the {\em $m$-coclique extension of $\Delta_1$} if $\Delta_2$ is the complement  $\overline{K}_m$ of the complete graph  on $m$ vertices.

Now we introduce the following special notation for Deza graphs.
Let $\Gamma$ be a Deza graph with parameters $(n, k, b, a)$, and let $v$ be a vertex of $\Gamma$.
We consider  the following subsets of $V(\Gamma)$:
\smallskip

$A(v)$ is the set of all vertices $u\in V(\Gamma)$ such that $|N(v)\cap N(u)| = a$.
\smallskip

$B(v)$ is the set of all vertices $u\in V(\Gamma)$ such that $|N(v)\cap N(u)| = b$.
\smallskip

$B[v] := B(v)\cup \{v\}$.
\smallskip

The number of vertices of $A(v)$ is denoted by $\alpha (v)$, and the number of vertices of $B(v)$ is denoted by $\beta (v)$.
It is known that the numbers $\alpha(v)$ and $\beta (v)$ for a given strictly Deza graph
$\Gamma$ are constants independent from the choice of a vertex $v$
(see \cite[Proposition~1.1]{EFHHH}). We denote these constants of  $\Gamma$ by $\alpha(\Gamma)$ and $\beta(\Gamma)$, respectively.

The main results of the present paper are the following two theorems.

\begin{theorem}
Let $\Gamma$ be a strictly Deza graph with parameters $(n, k, b, a)$ and  $\beta (\Gamma) > 1$.
The parameters $k$ and $b$ of $\Gamma$ satisfy the condition $k = b + 1$  if and only if $\Gamma$ is isomorphic to the $2$-clique extension of
the complete multipartite graph with parts  of size  ${\displaystyle\frac{n - k+1}{2}}$.
\end{theorem}

Graphs from the conclusion of Theorem 1 are pointed out in \cite{EFHHH} (see Example 2.4).

\begin{theorem} Let $\Gamma$ be a strictly Deza graph with parameters $(n, k, k-1, a)$ and
$\beta (\Gamma) > 1$.
Then $a = 2k - n $ and $\Gamma$ is recognizable by its parameters.
\end{theorem}

By Theorem 1 $\Gamma$ is isomorphic to the $2$-clique extension of
the complete multipartite graph with parts  of size  ${\displaystyle\frac{n - k+1}{2}}$. Hence,
$a = n - 2(n -k + 1) + 2 = 2k - n$ and $\Gamma$ has parameters $(n, k, k-1, 2k - n )$.

In the forthcoming paper, we will consider  strictly Deza graphs with parameters $(n, k, k-1, a)$  and $\beta (\Gamma) = 1$.

\section{Preliminary results}

Let $\Gamma$ be a strictly Deza graph with parameters $(n,k,b,a)$.
By definition, $V(\Gamma) = \{v\}\cup A(v)\cup B(v)$ for any vertex $v\in \Gamma$. Hence, it is easy to see that $n = 1 + \alpha(\Gamma) + \beta(\Gamma)$.

\begin{lemma}\label{lemma 1}
The following equality holds for $\beta (\Gamma)$:
$$\beta := \beta (\Gamma) = \frac{k(k - 1) - a(n - 1)}{b - a}.$$
\end{lemma}
\begin{proof}  See Proposition 1.1 in \cite{EFHHH}.
\end{proof}

\begin{lemma}\label{lemma 2} A strictly Deza graph with parameters $(n, k, k-1, k-2)$ does not exist.
\end{lemma}
\begin{proof}
Let $\Gamma$ be a strictly Deza graph with parameters $(n, k, k-1, k-2)$. Since $b-a=1$, then we have
$\beta = k(k-1) - (k-2)(n-1) > 0$ by Lemma~1.

Let $k=2$. Then $\Gamma$ is the cycle of length $n$, and since $\Gamma$ has diameter $2$ we have that
$n\in \{4, 5\}$.
But the cycles of length 4 and 5 are strongly regular graphs.
Thus, no cycle can be a strictly Deza graph.

Let $k > 2$.  Since  $\beta > 0$, the following inequality holds:
$$ n < \frac{k^2 - 2}{k-2} = k + 2 + \frac{2}{k - 2}.$$

If $k=3$, then $n < 8$. However, strictly Deza graphs having less than 8 vertices do not exist \cite{EFHHH}.

If $k > 3$, then $n\leq k+2$.
Since $\Gamma$ is not a complete graph, we have $n = k + 2$. Hence, for every vertex $v$, there exists a unique nonadjacent vertex $u$ such that all other vertices of $\Gamma $ are adjacent both to $v$ and $u$.
However, $b = k$ in this case, and we have a contradiction to the condition $b=k-1$.
\end{proof}

\section{Proof of Theorem 1}

Till the end of the proof, let $\Gamma$ be a strictly Deza graph with parameters $(n, k, b, a)$, where $k = b + 1$ and
$\displaystyle\beta = \frac{k(k - 1) - a(n - 1)}{b - a} > 1$.

At first, we consider all possibilities of mutual placement of sets $B(v)$ and $N(v)$ for an arbitrary vertex $v$ of $\Gamma$.

\begin{lemma}\label{lemma 3}
Let $v\in V(\Gamma)$. Then, for the set $B(v)$, one of the following statements holds:

$(1)$ $B(v) \cap N(v) = \emptyset$;

$(2)$ $B(v) \subset N(v)$;

$(3)$ $|B(v) \cap N(v)| = 1$.
\end{lemma}
\begin{proof}
Let us recall that if a vertex $u$ belongs to $B(v)$, then
$|N(u) \cap N(v)| = b = k - 1$. If $u\in N(v)\cap B(v)$, then $N(v) = \{u\}\cup (N(u) \cap N(v))$ and $u$ is adjacent to no vertex outside the closed neighbourhood $N[v]$. If $w \in B(v)\setminus N[v]$, then there exist a unique vertex of $N(v)$ nonadjacent to $w$.
Therefore, if neither $(1)$ nor $(2)$ hold, then $|N(v)\cap B(v)| = 1$.
\end{proof}

We assume further that a vertex $v$ of $\Gamma$ is {\em of type} $(A)$, $(B)$, or $(C)$ if $B(v)$ satisfies statement $(1)$, $(2)$, or $(3)$ of Lemma 3, respectively.

\begin{lemma}\label{lemma 4}
 For parameters of  $\Gamma$, the following inequalities hold:

$(1)$ $\alpha > 0$;

$(2)$  $b > a >0$.
\end{lemma}

\begin{proof}
Since $\Gamma$ is not a strongly regular graph,  $\alpha \neq 0$ and $b \neq a$.
Let $a = 0$. Since $\Gamma$ has diameter 2, $A(v)$ is contained in $N(v)$ for each vertex $v$ in $\Gamma$. Moreover, any two vertices of $A(v)$ are not adjacent.

Fix a vertex $u$ of $\Gamma$. If a vertex $u$ is of type $(A)$, then $B(u) \cap N(u) = \emptyset$. Hence, $ B(u) = \Gamma \setminus N[u]$ and $A(u) = N(u)$.
Therefore,  we have  $\alpha = k$ and $\beta = n - k - 1$. Note that, from the equalities $a = 0$ and $\alpha = k$,  we have $A(v) = N(v)$. Thus, $B(v) = \Gamma \setminus N[v]$ for each vertex $v$ of $\Gamma$. However, it is impossible since $\Gamma$ is not a strongly regular graph.

If $u$ is a vertex of  type $(B)$ or $(C)$ in $\Gamma$, then $|B(u) \cap N(u)| \neq \emptyset$.  Since $A(u)\subset N(u)$,  any vertex $v$ of $N(u)\cap B(u)$ is adjacent to any vertex of $A(u)$. This contradicts the assumption  $a = 0$.
\end{proof}

By Lemma 4, it is obvious that, for any two vertices $v$ and $u$ of $\Gamma$, the~intersection of their neighbourhoods $N(v) \cap N(u)$ cannot be empty.

In further Lemmas 5 -- 11, our aim is to investigate  properties of  vertices of type~$(A)$.

\begin{lemma}\label{lemma 5}
Let $x$ be a vertex of type $(A)$ in $\Gamma$. If  there exist  two distinct vertices $x_1$ and $x_2$ in $B(x)$ such that $N(x)\cap N(x_1) = N(x)\cap N(x_2)$, then $N(x)\cap N(x_1) = N(x)\cap N(x_i)$ for each vertex $x_i$ from $B(x)$.
\end{lemma}
\begin{proof}
Since $b = k - 1$, we have $N(x)\cap N(x_1) = N(x)\cap N(x_2) = N(x_1)\cap N(x_2)$.
Let $x_i$ be a vertex from $B(x)$ and $x_i \notin \{x_1, x_2\}$. If $N(x)\cap N(x_1) \neq N(x)\cap N(x_i)$, then $|N(x)\cap N(x_1) \cap N(x_i)| = k - 2$.
By Lemma~2, the parameter $a$ is not equal to $k - 2$. Therefore,  $|N(x_1)\cap N(x_i)| = |N(x_2)\cap N(x_i)| = k-1$. However, there is a unique  vertex in $N(x_i)\setminus N(x)$. Hence, we have $N(x_i)\setminus N(x)\subset N(x_1)\cap N(x_2)$. Thus, $|N(x_1)\cap N(x_2)| = k$.  A contradiction with $|N(x_1)\cap N(x_2)| = k-1$.
\end{proof}

\begin{lemma}\label{lemma 6}
Let a vertex $x$ be of type $(A)$ in $\Gamma$,
and let $B(x) = \{x_1,\dots , x_{\beta}\}$.
Then the subgraph induced on $B[x]$ in $\Gamma$  is a coclique of size $\beta +1$.
\end{lemma}
\begin{proof}
Let $x_i x_j$ be an edge of the subgraph induced on $B[x]$ in $\Gamma$. If $N(x)\cap N(x_i) \neq N(x)\cap N(x_j)$, then $|N(x_i)\cap N(x_j)| = k - 2 = a $, which is impossible by Lemma~2.
Hence, $N(x)\cap N(x_i) = N(x)\cap N(x_j)$. Since $x_j \in N(x_i)\setminus N(x)$, the equality $N(x)\cap N(x_i) = N(x_i)\cap N(x_j)$ holds. Hence, there exists a vertex $y$ of $N(x)$ such that $N(x) = \{y\}\cup (N(x_i)\cap N(x_j))$. Since $N(x) \cap B(x) = \emptyset$, we have $y\notin B(x)$ and $N(x)\setminus N[y]\neq \emptyset$.

Therefore, there exists  a vertex $v$ in $(N(x_i) \cap N(x_j))\setminus N[y])$.  Since $x$ is a vertex of type $(A)$, we have $v\notin B(x)$. In this case, $|N(v)\cap N(x)|=a$. Moreover, $N(v)\cap N(x)$ is contained in $N(x_i)\cap N(x_j)$.
Hence,  $N(v)\cap N(x_i)$ contains the set $N(v)\cap N(x)$ and the vertex  $x_j$. Therefore, $N(v)\cap N(x_i)$ contains $a + 1=b$ vertices.  But this contradicts Lemma~2.
\end{proof}

\begin{lemma}\label{lemma 7}
For each vertex $x$ of type $(A)$ and each vertex $x_i$ in $B(x)$, the equality $B[x] = B[x_i]$ holds.
Moreover, $N(x_i)\cap B(x_i) = \emptyset$. Thus, each vertex $x_i\in B(x)$ is of type  $(A)$.
\end{lemma}
\begin{proof}
Let $x_i\in B(x)$. It is clear that  $x\in B(x_i)$. For any vertex $x_j \in B(x)$  $(x_j\neq x_i )$, the inequality $|N(x_i)\cap N(x_j)|\geq |N(x)\cap N(x_i)\cap N(x_j)| \geq k - 2$ holds. By Lemma~2 $|N(x_i)\cap N(x_j)| > k - 2$. Therefore, $|N(x_i)\cap N(x_j)| = k - 1$ and $x_j\in B(x_i)$. This implies the equality $B[x] = B[x_i]$.

Moreover, by Lemma~6, the subgraph induced on $B[x]$ in $\Gamma$ is a coclique.  Since $B[x] = B[x_i]$, we have $N(x_i)\cap B(x_i) = \emptyset$. Thus, each vertex $x_i\in B(x)$ is of type~$(A)$.
\end{proof}

\noindent{\bf Remark.} By the choice of a vertex $x$, for any vertex $x_i\in B(x)$, there is  a unique vertex in $N(x)$ nonadjacent to $x_i$. At the same time, there is  a unique vertex  in $N_2 (x)$ adjacent to $x_i$.

\begin{lemma}\label{lemma 8} Suppose that $x$ is a vertex of type $(A)$, $y \in N(x)$, and $y$ is nonadjacent to a vertex
$x_i \in B(x)$. Let $z\in N_2 (x)$ be adjacent to $x_i$. Then the following statements hold:

$(1)$ $N(x)\cap N(y) = N(x)\cap N(z)$;

$(2)$  $B[v]$ is contained in $N(x)\cap N(y)$ for any vertex $v$ from $N(x)\cap N(y)$.
\end{lemma}

\begin{proof}
$(1)$ At first, since $x$ is a vertex of type $(A)$ and $y\notin B(x)$, we have $|N(y) \cap N(x)| = a$.  Then $N(x) = \{y\}\cup (N(x)\cap N(x_i))$.
Moreover, $N(x)\cap N(y) \subseteq N(x)\cap N(x_i)$. By Lemma~4(2), we have $N(x) \cap N(y) \neq \emptyset$. If $v\in N(x) \cap N(y)$, then $v$ is adjacent to the vertex $y$ and $a - 1$ vertices in $N(x_i)$. But $v$ must be adjacent to at least $a$ vertices in $N(x_i)$.
Hence, $v$ is adjacent to $z$. Therefore, $N(x) \cap N(y)$ is contained in $N(z) \cap N(x)$.
 By Lemma~6, $z\notin B(x)$. Therefore, $|N(z) \cap N(x)| = a$. This implies the required equality $N(x)\cap N(y) = N(x)\cap N(z)$.

$(2)$ Let $v \in N(x)\cap N(y)$ and $w \in B(v)$. By  statement $(1)$, we have  $\{x, y, z, x_i\}\subseteq N(v)$. Since $w\in B(v)$, $w$ is adjacent to at least three vertices in $\{x, y, z, x_i\}$.
Now the required assertion follows from the equality $N(x)\cap N(y) = N(x)\cap N(z)$ proved in the previous statement and from the equality $N(x_i)\cap N(y) = N(x_i)\cap N(z)$ obtained from it by applying Lemma~7.
\end{proof}

\begin{lemma}\label{lemma 9}
For any vertex $x$ of type $(A)$, one of the following statements holds:

 $(1)$  There is a unique vertex $y \in N(x)$ such that $\{y\} = N(x)\setminus N(x_i)$ for any $x_i \in B(x)$. Moreover,
$| \bigcup_{i=1}^{\beta}(N(x_i) \setminus N(x))| = \beta$.

$(2)$  There is a unique vertex $z\in N_2(x)$ such that
$\{z\} = N(x_i)\setminus N(x) $ for any $x_i \in B(x)$. Moreover, $|\bigcup_{i=1}^{\beta}(N(x) \setminus N(x_i))| = \beta$.
\end{lemma}

\begin{proof}
 $(1)$  Let  $x$ be a vertex of type $(A)$. Then  $B(x)$ is contained in $N_2(x)$.
Since $k = b + 1$, we have $|N(x) \setminus N(x_i)|=1$ and $|N(x_i) \setminus N(x)|=1$.
By the condition of Theorem 1, we have $\beta > 1$.

Let the number of vertices in $\bigcup_{i=1}^{\beta}(N(x) \setminus N(x_i))$ be less than $\beta$.
Then there exist two distinct vertices $x_i$ and $x_j $ of $B(x)$ such that the equality
$N(x) \cap N(x_i) = N(x) \cap N(x_j)$ holds. In this case, by Lemma~5,
$N(x)\cap N(x_i) = N(x)\cap N(x_l)$ for any vertex $x_l$ in $B(x)$. Therefore, there exists a vertex $y$ such that
$N(x) =\{y\}\cup N(x)\setminus N(x_i)$ for any $x_i$ in $B(x)$. Moreover, for distinct vertices $x_j$ and $x_l$ of $B(x)$, any two differences $N(x_j)\setminus N(x)$ and $N(x_l)\setminus N(x)$ cannot coincide. Otherwise,
$|N(x_j)\cap N(x_l)| = k$, but this contradicts the condition on the parameter $b$ of $\Gamma$. Hence, the number of vertices in the set
$\bigcup_{i=1}^{\beta}(N(x_i) \setminus N(x))$ is equal to the number of vertices in $B(x)$.
Thus, the condition  $(1)$  for $\Gamma$ holds.

$(2)$ Let the differences $N(x) \setminus N(x_i)$ be pairwise distinct for $i \in \{1,\dots , \beta \}$. Therefore,
$|N(x_i) \cap N(x_j)\cap N(x)| = k - 2$. Since $a \neq k-2$ by Lemma~2, we have $|N(x_i) \cap N(x_j)| = b$ for any two distinct vertices $x_i$ and $x_j$ of $B(x)$.
It follows that, for any pair $N(x_i)$ and $N(x_j)$, there exist a  unique common vertex $z \in N(x_i) \cap N(x_j) \cap N_2(x)$. Thus, condition $(2)$  holds for $\Gamma$ .
\end{proof}

Further, if the conclusion of statement   $(1)$ of Lemma~9 holds, let $N(x_i) \setminus N(x) = \{y_i\}$ and
$\bigcup_{i=1}^{\beta}(N(x_i) \setminus N(x)) = \{ y_1, \dots , y_{\beta}\}.$
If the conclusion of statement $(2)$ of Lemma~9 holds, let $N(x) \setminus N(x_i) = \{z_i\}$ and
$\bigcup_{i=1}^{\beta}(N(x) \setminus N(x_i)) = \{ z_1, \dots , z_{\beta} \}.$

\begin{lemma}\label{lemma 10}

If statement  $(1)$ of Lemma~9 holds,  then $B(y) = \{ y_1, \dots , y_{\beta} \}$ and $y$ is of type $(A)$.

If statement $(2)$ of Lemma~9 holds,  then $B(z) = \{ z_1, \dots , z_{\beta} \}$ and $z$ is of type~$(A)$.
\end{lemma}

\begin{proof}
Let  statement  $(1)$ of Lemma~9 hold.
Let $v$ be any  vertex from $N(y) \setminus N[x]$. The vertex $v$ cannot belong to $B(x)$ by the choice of $y$. Therefore, $|N(x)\cap N(v)| = a$. But $N(x)\cap N(v)$ contains $y$ and, hence, only $a - 1$ vertices from $N(x)\cap N(x_i)$ for each $x_i\in B(x)$. However,  $v$ must be adjacent to at least $a$ vertices in $N(x_i)$. Thus, $v$ must be adjacent to a vertex in $N(x_i)\setminus N(x) = \{y_i\}$. Hence, any vertex from $N(y) \setminus N(x)$ is adjacent to $y_i$. By Lemma~8(1) any vertex $y_i$ belongs to $B(y)$.

Since any vertex $y_i\notin B(x)$, we have $|N(y_i)\cap N(x)| = a$. Thus, $N(y_i)$ does not contain $y$ and  $N(y)\cap \{ y_1, \dots , y_{\beta} \} = \emptyset$. Therefore, $y$ is of type $(A)$.

Let statement $(2)$ of Lemma~9 hold. If we consider an arbitrary  vertex $v$ from $N(z_i) \setminus N[x]$, then similar arguments as in the previous case prove  the lemma.
\end{proof}

We proved that the vertices $y$ and $z$ in Lemma~10  are of type $(A)$. Further, for convenience, if a vertex $v$ has type $(A)$, then we denote by $v^{\star}$ a unique vertex described in Lemma~9.

By Lemma~10,  we have two possibilities for vertices of type $(A)$ in $\Gamma$.
Further, if a vertex satisfies statement  $(1)$ of Lemma~9, then we call it {\em a
vertex of type  $(A1)$}. If a vertex satisfies statement $(2)$ of Lemma~9, then we call it {\em a
vertex of type  $(A2)$}.

\begin{lemma}\label{lemma 11}
For any vertex of type $(A)$ of $\Gamma$, one of the following statements holds.

$(1)$ If $x$ is a vertex of type  $(A1)$, then $N(x)\setminus \{x^{\star}\}$  contains $B[v]$  for every vertex $v\in N(x)\setminus \{x^{\star}\}$.

$(2)$ If $x$ is a vertex of type  $(A2)$, then $N(x)\setminus B(x^{\star})$ contains $B[v]$ for every vertex $v\in N(x)\setminus B(x^{\star})$.
\end{lemma}
\begin{proof}
$(1)$ Let $x$ be a vertex of type  $(A1)$, and let $v$ be an arbitrary vertex from $N(x)\setminus \{x^{\star}\}$.
If $v \in N(x)\cap N(x^{\star})$, then $B[v]$ is contained in $N(x)\setminus \{x^{\star}\}$ by Lemma~8(2).
Let $v \in N(x)\setminus N[x^{\star}]$.
Thus, if $w$ is an arbitrary vertex from $B(v)$, then  $w\neq x^{\star}$ by Lemma~9(1). Moreover, if $w$ is not adjacent to $x$, then $w$ is adjacent to all vertices of $B(x)$. Since $\beta > 1$,  there exist two distinct vertices $x_i$ and $x_j$ in $B(x)$ such that $w$ belongs to $N(x_i) \cap N(x_j)$.
But $N(x_i) \cap N(x_j)\subset N(x)$ since $x$ is a vertex of type  $(A1)$. A contradiction. Thus, $w$ is adjacent to $x$ and  $B[v]$ is contained in $N(x)\setminus \{x^{\star}\}$.

$(2)$ Let $x$ be a vertex of type  $(A2)$ and $v$ be an arbitrary vertex from $N(x)\setminus B(x^{\star})$.
Then $v$ is adjacent  to  all vertices of $B[x]$.
Let there exist a vertex $w$ in $B(v)$ that does not belong to $N(x)\setminus B(x^{\star})$.

If $w$ is not adjacent to $x$, then $w$ is adjacent to a vertex
$x_i$ from $B(x)$.
Hence, by Lemma~9(2), $w\in N(x_i) \setminus N(x) = \{x^{\star}\}$.
Since  $w\in B(v)$, we have $v\in B(w)=B(x^{\star})$.
A contradiction with the choice of $v$ from $N(x)\setminus B(x^{\star})$.
Thus, $w$ is adjacent to $x$ and  $B[v]$ is contained in $N(x)\setminus N(x^{\star})$.

If $w\in N(x)\cap B(x^{\star})$, then
here exist a unique vertex $x_i\in B(x)$ such that $w\notin N(x_i)$. However, $|N(x^{\star})\cap N(w)|= |N(v)\cap N(w)| = k-1$.
Thus, $|N(x^{\star})\cap N(v)|\geq k-2$. It contradicts $v\notin  B(x^{\star})$ and Lemma~2.
Hence, $w\in N(x)\setminus B(x^{\star})$ and $B[v]\subseteq N(x)\setminus B(x^{\star})$.
\end{proof}

In the next lemma, we study vertices of types $(B)$ and $(C)$ in $\Gamma$.

\begin{lemma}\label{lemma 12}
$(1)$ Let $v$ be a vertex of type $(B)$ in $\Gamma$, and let $u$ be any vertex from $B(v)$. Then  $B[u] = B[v]$
and $u$ is of type $(B)$.

$(2)$ Let $v$ be a vertex of type $(C)$ in $\Gamma$, and let $u$ be any vertex from $B(v)$. Then  $B[u] = B[v]$
and $u$ is of type $(C)$.
\end{lemma}
\begin{proof}
$(1)$  If $v$ is a vertex of type $(B)$ in $\Gamma$, then,  for each $u \in B(v)$, we have $N[u] = N[v]$ by the definition of $B[v]$.
Therefore, $B[u]$ contains each vertex from $B[v]$. In view of Lemma~1, $|B[v]| = \beta = |B[u]|$, and we have the equality $B[u] = B[v]$.
Hence, $B[u]\subseteq N[u]$ and $u$ is of type $(B)$ for any $u \in B[v]$.

$(2)$  Let $v$ be a vertex of type $(C)$ in $\Gamma$.  Let $u\in B(v)$ and $B[u] \neq B[v]$.
By the definition of a vertex of type $(C)$,  there exists a unique vertex $v^{\star} \in N(v)\cap B(v)$ such that
$N[v] = N[v^{\star}]$.  Thus, $u\neq v^{\star}$ and $u\notin N[v]$. Furthermore, $N(u)\cap N(v) = N(v)\setminus \{v^{\star}\}$.

By Lemma~7 and by the previous statement,  a vertex $u$ can not be of types $(A)$ and $(B)$. Hence, $u$ is of type $(C)$
and there exists $u^{\star} \in N(u)\cap B(u)$ such that $N[u] = N[u^{\star}]$.

Let $w\in B[u]\setminus B[v]$. Clearly, $u^{\star} \in B(v)$ and, therefore, $w\neq u^{\star}$. Then
$N(w)\cap N(u) = N(u)\setminus \{u^{\star}\} = N(u)\cap N(v) = N(v)\setminus \{v^{\star}\}$. But this contradicts  $w\notin B[v]$. Hence, $B[u] = B[v]$ and $u$ is a vertex of type $(C)$.

\end{proof}
\begin{lemma}\label{lemma 13}
For any vertices $v$ and $u$ of $\Gamma$, if $B[v] \cap B[u] \neq \emptyset$, then $B[v] = B[u]$.
\end{lemma}
\begin{proof}
This lemma is a direct corollary of Lemmas~7 and 12.
\end{proof}
\begin{lemma}\label{lemma 14}
 If $v$ is a vertex of type $(B)$ or $(C)$ of $\Gamma$, then for any vertex $u \in N(v)\setminus B(v)$ we have
 $ B[u]\subseteq N(v)\setminus B(v)$.
\end{lemma}
\begin{proof}
If $v$ is a vertex of type $(B)$ or $(C)$ in $\Gamma$, then we have
$N[v] = N[v']$ for any $v'\in N[v]\cap B[v]$. Let $u \in N(v)\setminus B(v)$. In this case,  $N(u)$ contains $N[v]\cap B[v]$.  Let $u' \in B(u)$ and $u' \notin N(v)\setminus B(v)$.  By Lemma~13, $u' \notin B[v]$. Therefore, $u' \notin N[v]$. But $u'$ is adjacent to $k - 1$ vertices from $N(u)$. Since $|N[v]\cap  B[v]| > 1$, we find that $N(u')\cap B(v) \neq \emptyset$ and $u'$ is adjacent to some vertex $v'\in B[v]$. This fact contradicts $N[v] = N[v']$  for any $v'\in N[v]\cap B[v]$.
\end{proof}
\begin{lemma}\label{lemma 15} Let $v$ be a vertex of $\Gamma$.
Then the following statements hold.

$(1)$ If $v$ is of type  $(A1)$ or $(C)$ in $\Gamma$, then $\beta + 1$ divides $k - 1$.

$(2)$ If $v$ is of type  $(A2)$ or $(B)$ in $\Gamma$, then $\beta + 1$ divides $k - \beta$.
\end{lemma}
\begin{proof}
By Lemma~13, the sets
$B[v]$ and $B[u]$ either coincide or do not intersect for any two vertices $u$ and $v$.
Since $|B[v]| = |B[u]| = \beta + 1$,  statements  $(1)$ and  $(2)$ follow from Lemma~11 and Lemma~14.
\end{proof}

\begin{lemma}\label{lemma 16}
Types of all the vertices of $\Gamma$ satisfy only statement $(1)$ or only statement $(2)$ of Lemma~15.
\end{lemma}
\begin{proof}
If   $\Gamma$ contains vertices of types  satisfying different statements of
Lemma~15, then $\beta + 1$ divides both $k - 1$ and $k - \beta$. But this contradicts $\beta > 1$.
\end{proof}
\begin{lemma}\label{lemma 17}
A graph $\Gamma$ all of whose vertices are of type $(B)$ does not exist.
\end{lemma}
\begin{proof}
 Let all the vertices of $\Gamma$ be of type $(B)$ and $v \in V(\Gamma)$. By Lemma~14, if $u \in N(v)\setminus B(v)$, then the set  $B[u]$ is contained in $N(v)\setminus B(v)$.

For any $u$ in $(N(v_1)\setminus B(v_1))\cap (N(v_2)\setminus B(v_2))$ the set $B[u]$ is contained in
$(N(v_1)\setminus B(v_1))\cap (N(v_2)\setminus B(v_2))$.

If $v_1$ is adjacent to $v_2$, then $N(v_1)\cap N(v_2)$ contains $B(v_1)$, $B(v_2)$, and $B[u]$ for any $u\in (N(v_1)\setminus B(v_1))\cap (N(v_2)\setminus B(v_2))$. Hence, by Lemma~13, $\beta +1$ divides $|(N(v_1)\setminus B(v_1))\cap (N(v_2)\setminus B(v_2))| = a - 2\beta$.

If $v_1$ is not adjacent to $v_2$, then $(N(v_1)\cap N(v_2))\cap (B(v_1)\cup B(v_2)) = \emptyset$.
Hence, by Lemma~13,  $\beta +1$ divides $|N(v_1)\cap N(v_2)| = a$.

Since $\Gamma$ is a strictly Deza graph, we have both possibilities for $v_1$ and $v_2$ to be adjacent or to be nonadjacent.
Comparing these two cases, we have $\beta +1$ divides $2\beta$.
However, this contradicts the condition $\beta > 1$.
\end{proof}
\begin{lemma}\label{lemma 18}
 There are no vertices of type  $(A1)$ in $\Gamma$.
\end{lemma}
\begin{proof} Let a vertex $x$ be of type  $(A1)$ in $\Gamma$.  Fix the vertex $x$. In view of Lemma~16,
each vertex $u$ in $\Gamma$ is either of type   $(A1)$ or of type $(C)$.

By Lemma~9, not only for $x$ but for each vertex $u$ of $\Gamma$ of type  $(A1)$, we have
 a unique vertex $u^{\star}\in N(u)$ such that $N(u^{\star})\cap B(u) = \emptyset$. Moreover,  by Lemma~11(1),
   for any vertex $w\in N(u)\setminus \{u^{\star}\}$ we have $B[w]$ is contained in $N(u)\setminus \{u^{\star}\}$.

By the definition, if a vertex $u$ of $\Gamma$ is of type $(C)$, then $|N(u)\cap B(u)|=1$.
We  denote the vertex from $N(u)\cap B(u)$ by $u_{\star}$. In this case, we have $N[u] = N[u_{\star}]$.
Moreover,  by Lemma~14, for any vertex $w\in N(u)\setminus \{u_{\star}\}$ we have $B[w]$
is contained in $N(u)\setminus \{u_{\star}\}$.

By Lemma~4(1), $N(x)\cap N(x^{\star})\neq \emptyset$. Let  $u\in N(x)\cap N(x^{\star})$.
Seeing the vertex $x$ is of type  $(A1)$, we have  $u\notin B[x]$ and $|N(x)\cap N(u)|=a$.

If $u$ is of type $(C)$, then for each vertex $w \in (N(x) \cap N(u))\setminus \{x^{\star}, u_{\star}\}$ the set $B[w]$ is
contained in $(N(x) \cap N(u))\setminus \{x^{\star}, u_{\star}\}$. Hence, $|B[w]| = \beta +1$ devides
$|(N(x) \cap N(u))\setminus \{x^{\star}, u_{\star}\}|$ which is equal to $a-2$.

If $u$ is of type  $(A1)$, then  for each vertex $w \in (N(x) \cap N(u))\setminus \{x^{\star}, u^{\star}\}$ the set $B[w]$ is
contained in $(N(x) \cap N(u))\setminus \{x^{\star}, u^{\star}\}$. Hence, $|B[w]| = \beta +1$ devides
$|(N(x) \cap N(u))\setminus \{x^{\star}, u^{\star}\}|$
which is equal to $a-2$ if $u^{\star}$ is adjacent $x$ or is equal to $a-1$ if $u^{\star}$ is nonadjacent $x$.

By lemma 8(2), $B[w]$ is contained in $N(x) \cap N(x^{\star})$ for any vertex $w$ from $N(x) \cap N(x^{\star})$.
Hence,   by lemma 13, $|B[w]| = \beta +1$ devides $|N(x) \cap N(x^{\star})| = a$.
On the other hand,  $\beta +1$ devides $a-2$ or $a-1$.
However, this contradicts  $\beta > 1$.
\end{proof}

\begin{lemma}\label{lemma 19}
There are no vertices of type  $(A2)$ in $\Gamma$.
\end{lemma}
\begin{proof}
Let $x$ be a vertex of type  $(A2)$ in $\Gamma$. Fix the vertex $x$.  In this case, by Lemma~16, any vertex $v$ of
$\Gamma$ is of type  $(A2)$ or $(B)$.

By Lemmas~9 and 11(2), not only for $x$ but for each vertex $u$ of $\Gamma$ of type  $(A2)$, we have
 only vertex $u^{\star}\notin N(u)$ such that $B(u^{\star})\subseteq N(u)$. Moreover,  by Lemma~11(2)
   for any vertex $w\in N(u)\setminus B(u^{\star})$ we have $N(u)\setminus B(u^{\star})$  contains $B[w]$.

By the definition, if a vertex $u$ of $\Gamma$ is of type $(B)$, then $B(u)\subseteq N(u)$.
Moreover,  by Lemma~14 for any vertex $w\in N(u)\setminus B(u)$ we have $N(u)\setminus B(u)$
contains $B[w]$.

By Lemma~4(1), $N(x)\cap N(x^{\star})\neq \emptyset$. Let  $u\in N(x)\cap N(x^{\star})$.
Seeing the vertex $x$ is of type  $(A2)$, in view of Lemma~8, we have  $u\notin B[x]$, $|N(x)\cap N(u)|=a$, and $B(x^{\star})\subseteq N(x) \cap N(u)$.

If $u$ is of type $(B)$, then $(B(x^{\star})\cup B(u))\subseteq (N(x) \cap N(u))$.
Moreover, for each vertex $w \in (N(x) \cap N(u))\setminus (B(x^{\star}) \cup B(u))$ the set $B[w]$ is
contained in $(N(x) \cap N(u))\setminus (B(x^{\star})\cup B(u))$. Hence, $|B[w]| = \beta +1$ devides
$|(N(x) \cap N(u))\setminus (B(x^{\star}) \cup B(u))|$ which is equal to $a-2\beta$.

If $u$ is of type  $(A2)$, then $B(x^{\star})\subseteq N(x) \cap N(u)$.
By Lemma~11(2), $B(u^{\star})\subseteq N(x) \cap N(u)$ or $B(u^{\star})\cap N(x) \cap N(u) = \emptyset$.
Moreover, for each vertex $w \in (N(x) \cap N(u))\setminus (B(x^{\star}) \cup B(u^{\star}))$ the set $B[w]$ is
contained in $(N(x) \cap N(u))\setminus (B(x^{\star}) \cup B(u^{\star}))$. Hence, by Lemma~13,  $|B[w]| = \beta +1$ devides
$|(N(x) \cap N(u))\setminus (B(x^{\star}) \cup B(u^{\star}))|$,
which is equal to $a-2\beta $ or $a-\beta $.

By lemmas 8(2)  and 13, $\beta +1$ devides $|N(x) \cap N(x^{\star})| = a$.
On the other hand, $\beta +1$ devides $a-2\beta$ or $a-\beta $.
However, this contradicts  $\beta > 1$.
\end{proof}

\begin{lemma}\label{last}
If all vertices in $\Gamma$ are of  type $(C)$, then $\Gamma$ is isomorphic to the $2$-clique extension of a complete multipartite graph with parts of size ${\displaystyle\frac{n - k+1}{2}}$.
\end{lemma}
\begin{proof}
Let each vertex in $\Gamma$ be of type $(C)$.
Consider the relation $\rho$ on the set of all vertices of $\Gamma$
setting $v\ \rho \ u$ if and only if $u \in B[v]$ in $\Gamma$.

By Lemma~13, the relation $\rho$ is an
equivalence relation.

Let us consider a quotient graph  $\Gamma / \rho$ setting $B[v]$ adjacent to  $B[u]$ if and only if $v$ adjacent to $u$.
 By Lemma~14, the map $v\rightarrow B[v]$ from $\Gamma$ on $\Gamma / \rho$ preserves the adjacency of vertices.

The quotient graph $\Gamma_{\rho} = \Gamma / \rho$ has the vertex set $\{B[v] | v \in \Gamma\}$. Thus,
$\Gamma_{\rho}$ has exactly
$\displaystyle\frac{n}{\beta + 1}$ vertices. Moreover, in view of Lemma~14, the degree of each vertex is equal to
$\displaystyle\frac{k - 1}{\beta + 1}$.

Let $v_{\rho}=B[v]$ and $u_{\rho}=B[u]$ be distinct vertices of $\Gamma_{\rho}$.
According to the equivalence relation $\rho$, there are only two possibilities for the number of vertices in  $N(v_{\rho})\cap N(u_{\rho})$ in $\Gamma_{\rho}$. We use $\{v^{\star}, u^{\star}\}$ in the same sense as well as in the proof of Lemma~18.

$(1)$ Let a vertex $u$ do not belong to $N[v]\cup B[v]$.
In this case, $N(v)\cap N(u)$ contains $B[w]$ for any
$w\in (N(v)\cap N(u))\setminus \{v^{\star}, u^{\star}\}$. But $N(v)\cap N(u)$ does not contain $v^{\star}$ and $u^{\star}$.
Hence, the vertices $v_{\rho}$ and $u_{\rho}$ have exactly $\displaystyle\frac{a}{\beta + 1}$ common adjacent vertices in $\Gamma_{\rho}$.

$(2)$ Let a vertex $u$ belong to $N[v]\setminus B[v]$.
In this case, $N(v)\cap N(u)$ contains both $v^{\star}$ and $u^{\star}$, and $B[w]$ for any
$w\in (N(v)\cap N(u))\setminus \{v^{\star},u^{\star}\}$.
Hence, the vertices $v_{\rho}$ and $u_{\rho}$ have exactly $\displaystyle\frac{a - 2}{\beta + 1}$ common adjacent vertices in $\Gamma_{\rho}$.

Since $\beta > 1$, these two cases can not be implemented together.

If case $(1)$ occurs for the vertices of $\Gamma$, then it follows that $N[v] \setminus B[v] = \emptyset$.
However,
$N[v] \setminus B[v] = N(v) \setminus \{v^\star\}$, and so $k = 1$, a contradiction,
because $\Gamma$ is a strictly Deza graph and has diameter 2.

Suppose that case  $(2)$ occurs  for the vertices of $\Gamma$.
Then the quotient graph $\Gamma_{\rho}$ is a
complete graph on $\displaystyle\frac{n}{\beta + 1}$ vertices
and $B[v]$ is the union of $\frac{(\beta+1)}{2}={\displaystyle\frac{n - k+1}{2}}$ disjoint edges.

Hence that graph $\Gamma$ is a extension of the complete graph on $\displaystyle\frac{n}{\beta + 1}$ vertices with $t$ disjoint copies of $K_2$, there $t = \displaystyle\frac{(\beta + 1)}{2}$. Thus, $\Gamma$ is isomorphic to the $2$-clique extension of a complete multipartite graph with parts of size
$t= \displaystyle\frac{(\beta + 1)}{2} = \displaystyle\frac{(n - k+1)}{2}$.
\end{proof}

We proved that if a strictly Deza graph $\Gamma$ satisfies the condition of Theorem 1, then $\Gamma$ cannot contain only vertices of type $(B)$ by Lemma~17. There are no vertices of type  $(A1)$  in $\Gamma$ by Lemma~18 and there are no vertices of type  $(A2)$  in $\Gamma$ by Lemma~19.

Hence, all the vertices in $\Gamma$ are of type $(C)$. By Lemma~\ref{last},  $\Gamma$ satisfies the conclusion of Theorem~1.

Theorem~1 is proved.

\end{document}